\documentclass[11pt,twoside]{amsart}
\usepackage[all]{xy}
        \usepackage {amssymb,latexsym,amsthm,amsmath}

\long\def\symbolfootnote[#1]#2{\begingroup%
\def\thefootnote{\fnsymbol{footnote}}\footnote[#1]{#2}\endgroup}

\newcommand{\diag}{\textup{diag}}

\newcommand{\F}{\mathbb{F}}

\DeclareMathOperator{\lcm}{lcm}

\DeclareMathOperator{\Sym}{Sym}
\DeclareMathOperator{\bd}{bd}
\DeclareMathOperator{\PSL}{PSL}
\DeclareMathOperator{\GL}{GL}
\DeclareMathOperator{\SL}{SL}
\DeclareMathOperator{\PSU}{PSU}
\DeclareMathOperator{\SU}{SU}
\DeclareMathOperator{\U}{U}
\makeatletter
\def\imod#1{\allowbreak\mkern10mu({\operator@font mod}\,\,#1)}
\makeatother

\newtheorem{theorem}{Theorem}[section]
\newtheorem{lemma}[theorem]{Lemma}

\newtheorem{proposition}[theorem]{Proposition}
\newtheorem{conjecture}[theorem]{Conjecture}
\newtheorem{problem}[theorem]{Problem}
\theoremstyle{definition}

\newtheorem{example}[theorem]{Example}
\numberwithin{equation}{section}

\newcommand{\ignore}[1]{}

\newcommand{\mynote}[1]{}
\begin{document}
\setcounter{section}{0}
\title{ On Shalev's conjecture for type $A_n$ and ${}^{2}A_n$}
\author{Alexey Galt}
\address{Sobolev Institute of Mathematics, Novosibirsk, Russia}
\email{galt84@gmail.com}
\author{Amit Kulshrestha}
\address{IISER Mohali, Knowledge City, Sector 81, Mohali 140 306 India}
\email{amitk@iisermohali.ac.in}
\author{Anupam Singh}
\address{IISER Pune, Dr. Homi Bhabha Road, Pashan, Pune 411 008 India}
\email{anupamk18@gmail.com}
\author{Evgeny Vdovin}
\address{Sobolev Institute of Mathematics, Novosibirsk, Russia}
\email{vdovin@math.nsc.ru }
\thanks{This work is part of DST-RFBR joint project supported on India side by grant INT/RUS/RFBR/P-288 and on Russia side by grant RFBR according to the research project ¹17-51-45000}
\today


\begin{abstract}
In the paper we consider images of finite simple projective special linear and unitary groups under power words. In
particular, we show that if $G\simeq \PSL_n^\varepsilon (q)$, then for every power words of type $x^M$ there exist
constant $c$ and $N$ such that $\vert \omega(G)\vert >c\frac{\vert G\vert }{n}$ whenether $\vert G\vert >N$.
\end{abstract}

\maketitle

\section{Introduction}
Recently there has been considerable interest in word maps on finite, algebraic and topological groups see, for
example, \cite{bz,gkp,glost,gt,hls,kt,ls,np} and a survey \cite{sh}. Recall that a group word
$\omega=\omega(x_1,\ldots,x_d)$ is an element of a free group $F_d$ with free generators $x_1,\ldots,x_d$. We may write
$\omega=x_{i_1}^{m_1}\ldots x_{i_k}^{m_k}$ where $i_j\in\{1,\ldots,d\}$, and $m_j$ are integers. We say that $\omega$
is a power word, if $\omega=\upsilon^m$ for $m>1$ and some word $\upsilon$.
For a group $G$ and $g_1,\ldots,g_d\in G$ we write
$$\omega(g_1,\ldots,g_d)=g_{i_1}^{m_1}\ldots g_{i_k}^{m_k} \in G.$$
The corresponding map $\omega: G^d\rightarrow G$ is called a word map and its image is denoted by $\omega(G)$.  Many
papers are devoted to estimate the size of $\omega(G)$. In~\cite{ls},  Larsen and Shalev proved the following

\begin{theorem}\label{LSth}
Let $G$ be a finite simple group of Lie type and of rank $n$ and $\omega\neq1$ be a word.
Then there exists $N=N(\omega)$ such that if $G$ is not of type $A_n$ or ${\ }^2A_n$, and $|G|\geqslant N$ then
$$|\omega(G)|\geqslant cn^{-1}|G|$$
for some absolute constant $c >0$.
\end{theorem}

Notice that $c$ depends on $w$ in Theorem \ref{LSth}. Later in~\cite{np},
Nikolov and Pyber found a weaker lower bound for the groups of type $A_n$ and
${\ }^2A_n$. In the expository article~\cite{sh}, Shalev conjectured that the
Theorem~\ref{LSth} holds for all groups of Lie type.

\begin{conjecture}{\cite[Conjecture 5.6]{sh}}. For every word $\omega\neq1$ there exists a number $N = N(\omega)$
such that if $G$ is an alternating group of degree $n$ or a finite simple group
of Lie type of rank $n$, and $|G|\geqslant N$, then
$$|\omega(G)|\geqslant cn^{-1}|G|,$$
where $c > 0$ is an absolute constant.
\end{conjecture}

Further Shalev states as a problem that in case of non-power words a stronger
conclusion might hold.
\begin{problem}{\cite[Problem 5.7 (i)]{sh}}. Suppose $\omega$ is not a power word, and let $G$ be a finite simple group.
Is it true that there exist $N,c > 0$ depending on $\omega$ such that if $|G|\geqslant N$ then $|\omega(G)|\geqslant
c|G|$?
\end{problem}
This problem shows that power words are of considerable interest and in this case one could expect the lowest bound on
$|\omega(G)|$.  In this article, we explore Shalev conjecture for finite groups of type $A_n$ and ${\ }^2A_n$ and for
power words $\omega=x^M$. It is clear, that if $\omega=x_{i_1}^{m_1}\ldots x_{i_k}^{m_k}$ and $m_1+\ldots+m_k\not=0$,
then $\vert \omega(G)\vert\geqslant \vert \bar{\omega}(G)\vert$, where $\bar{\omega}=x^{m_1+\ldots +m_k}$. In the paper
we restrict ourself by the consideration of power words of type $\omega=x^M$.

Throughout the paper we assume that
$$\omega=x^M$$ and fix symbols $\omega$ and $M$ for this situation.

First we show by an example that the constant $c$ in the conjecture depends on the word $\omega$. Next, we prove that
for power words $\omega=x^M$ the conjecture is true (see Theorem \ref{1torus}). Moreover  we improve the bound as follows.

\begin{theorem}\label{th1}
Let $\omega=x^M$ be a power word and $G=\PSL_n^\varepsilon (q)$. Then  there
exist positive constant $N$ depending only on $M$ such that if $|G|\geqslant
N$, then
$$|\omega(G)| \geqslant \frac{\ln(n)}{2n\cdot M^2} |G|.$$
\end{theorem}

We believe, that for  the estimate in Theorem \ref{th1} it is possible to
prove  a better bound $|\omega(G)| \geqslant c\frac{\ln^k(n)}{n} |G|$, where
$k$ is any fixed power. However,   we cannot remove $n^{-1}$ in the
conjecture in view of the following theorem.

\begin{theorem}\label{th2}
Let $\omega=x^M$ be a power word, where $M=q-1$. Then for any  $N $, $c$
there exists $G=\PSL_n(q)$ such that $|G|\geqslant N$ and
$$|\omega(G)|< c|G|.$$
\end{theorem}

Notice also that in view of Theorem \ref{1torus}, there exists $N$ such that if $G=\PSL_n^\varepsilon(q)$ and $\vert
G\vert \geqslant N$, then $\vert \omega(G)\vert\geqslant \frac{\vert G\vert}{2nM}$, and this bound could be better than
the bound in Theorem~\ref{th1}.

\section{Preliminaries}
Out notation is standard and in general agree with that of \cite{Isaacs} and
\cite{ATLAS}. For integers $m_1,\dots,m_s$ by $(m_1,\dots,m_s)$ or
$\gcd\{m_1,\dots,m_s\}$ we denote their greatest common divisor, and by
$[m_1,\dots,m_s]$ or $\lcm\{m_1,\dots,m_s\}$ we denote their least common
multiple. We write $H\leq G$ and $H\unlhd G$, if $H$ is a subgroup or $H$ is
a normal subgroup of $G$ respectively. We recall some relations about $\gcd$.

\begin{lemma}\label{Zav}{\cite[Lemma 6(iii)]{za}}
Let $a,s,t\in\mathbb{N}$. Then
\begin{itemize}
  \item[{\em (a)}] $(a^s-1,a^t-1)=a^{(s,t)}-1$,
  \item[{\em (b)}] $(a^s+1,a^t+1)=
  \begin{cases} a^{(s,t)}+1 & \text{if both } s/(s,t) \text{ and } t/(s,t) \text{ are odd},\\
                (2,a+1) & \text{otherwise,}
  \end{cases}$
  \item[{\em (c)}] $(a^s-1,a^t+1)=
  \begin{cases} a^{(s,t)}+1 & \text{if } s/(s,t) \text{ is even and } t/(s,t) \text{ is odd},\\
                (2,a+1) & \text{otherwise.}
  \end{cases}$
\end{itemize}
\end{lemma}

Recall that any finite abelian group $A$ can be uniquely presented as a
product of cyclic groups $A=\mathbb{Z}_{d_1}\times\ldots\times
\mathbb{Z}_{d_k}$, where $\mathbb{Z}_{d_i}$ is a cyclic group of order $d_i$
and $d_1$ divides $d_2$, $d_2$ divides $d_3$,\ldots, $d_{k-1}$ divides $d_k$.
We use the notation $A=d_1\times\ldots\times d_k$ and refer to it as a {\em
standard decomposition} of $A$ for brevity.

By $p$ we always denote some prime number and $q$ is a power of $p$, ${\F}_q$ denotes a finite field of $q$ elements,
and $\overline{\F}_q$ is  an algebraic closure of ${\F}_q$. The symmetric group on $n$ elements is denoted by $\Sym_n$.
By $\diag(\lambda_1,\lambda_2,\ldots,\lambda_n)$ we denote a diagonal  $n\times n$-matrix with
$\lambda_1,\lambda_2,\ldots,\lambda_n$ on the diagonal. By $\bd(T_1,T_2,\ldots,T_n)$ we denote a block-diagonal matrix
with square blocks $T_1,T_2,\ldots,T_n$ on the diagonal.

Let $\overline{G}$ be a connected  simple algebraic group over an algebraic closure $\overline{\F}_p$ of a finite field
$\F_p$. We say that $\sigma$ is a {\em Steinberg endomorphism} (or a {\em Frobenius map}), if the set
$\overline{G}_\sigma$ of $\sigma$-stable points of $\overline{G}$ is finite. Any group $G=O^{p'}(\overline{G}_\sigma)$
is called a finite group of Lie type. If $\overline{G}$ is of type $A_{n}$, then $O^{p'}(\overline{G}_\sigma)$ is a
homomorphic image of either $\SL_{n+1}(q)$ or $\SU_{n+1}(q)$ for some $q$ and we say that $G$ is a group
$A_n^\varepsilon(q)$ in this case, where $\varepsilon\in\{+,-\}$ and $A_n^+(q)$ is a homomorphic image of
$\SL_{n+1}(q)$, while $A_n^-(q)$ is a homomorphic image of $\SU_{n+1}(q)$. Recall that an element is regular if its
centralizer in the corresponding group over the algebraic closure of ${\F}_q$ is equal to the Lie rank of the group
$G$. In particular, a semisimple element is regular, if the connected component of its centralizer is a maximal torus.
Thus every regular semisimple element belongs to exactly  one maximal torus. Recall also that every semisimple element
of $\overline{G}_\sigma$ lies in a $\sigma$-stable maximal torus $\overline{T}$ of~$\overline{G}$.

If $\overline{T}$ is a maximal $\sigma$-stable torus of $\overline{G}$, then
$T=\overline{T}\cap G$ is called a {\em maximal torus} of $G$ and
$N(G,T)=N_{\overline{G}}(\overline{T})\cap G$ is called an {\em algebraic
normalizer}. Clearly $N(G,T)\leq N_G(T)$, and the next lemma gives a
sufficient condition for the equality.

\begin{lemma}\label{NomMaxTorus}
Let $T$ be a maximal torus of a finite group of Lie type
$G=O^{p'}(\overline{G}_\sigma)$. Assume also that $T$ contains a regular
element $t$. Then $N(G,T)=N_G(T)$.
\end{lemma}

\begin{proof}
Since $t$ is regular, we obtain $\overline{T}\leq C_{\overline{G}}(T)^0\leq
C_{\overline{G}}(t)^0=\overline{T}$. Now $C_{\overline{G}}(T)^0\unlhd
N_{\overline{G}}(T)$, so $N_{\overline{G}}(T)\leq
N_{\overline{G}}(\overline{T})$ and the lemma follows.
\end{proof}

We denote the Weyl group $N_{\overline{G}}(\overline{T})/\overline{T}$ of $A_{n-1}$ by $W$. Since all maximal tori are
conjugate, $W$ does not depend on a particular choice of $\overline{T}$ and is isomorphic to $\Sym_n$. If we fix a
$\sigma$-stable torus $\overline{H}$ of $\overline{G}$, then $\sigma$ leaves $N_{\overline{G}}(\overline{H})$
invariant, and so induces an automorphism on $W=N_{\overline{G}}(\overline{H})/\overline{H}$ and we denote this
automorphism also by $\sigma$. We say that $u,w\in W$ are $\sigma$-conjugate, if there exists $v\in W$ such that
$u=v^{-1}wv^\sigma$. There is a bijection between the $G$-classes of $\sigma$-stable maximal tori of $\overline{G}$ and
the $\sigma$-conjugacy classes of $W$, (see~\cite[Propositions 3.3.1, 3.3.3]{Car}). Moreover, by~\cite[Proposition
3.3.6]{Car} if a  torus $T=\overline{H}^g\cap G$ corresponds to an element $w$ of $W$, we have
$$N(G,T)/T\simeq  C_{W,\sigma}(w)=\{x\in W\mid  x^{-1}wx^\sigma =w\}.$$

Given $w\in W$, a representative of the corresponding conjugacy class of maximal tori we denote by $T_w$, and we denote
$N(G,T)/T$ by $W(T)$.  If $G=A_n^+(q)$, then we may choose $\overline{H}$ so that $\sigma$ acts trivially on
$N_{\overline{G}}(\overline{H})/\overline{H}$ (in this case $\overline{H}_\sigma$ is of order $(q-1)^n$), so in this
case there is a bijection between the $G$-classes of $\sigma$-stable maximal tori of $\overline{G}$ and conjugacy
classes of $W$. If $G=A_n^-(q)$, then we also may choose $\overline{H}$ so that $\sigma$ acts trivially on
$N_{\overline{G}}(\overline{H})/\overline{H}$ (in this case $\overline{H}_\sigma$ is of order $(q+1)^n$), so in this
case there is a bijection between the $G$-classes of $\sigma$-stable maximal tori of $\overline{G}$ and conjugacy
classes of $W$. In both cases, we obtain that for a maximal torus $T=T_w$ the identity $W(T)\simeq C_W(w)$ holds.

Each element of $\Sym_n$ can be uniquely expressed as a product of disjoint
cycles. Lengths of these cycles define a set of integers, which is called the
cycle-type of the element. Two permutations are conjugate if and only if they
have the same cyc\-le-ty\-pe. Thus the conjugacy classes of $\Sym_n$ are in
one to one correspondence with the partitions of $n$, and so the conjugacy
classes of maximal tori in $A_{n-1}^\varepsilon(q)$ is in one to one
correspondence with the partitions of $n$.

The structure of maximal tori in linear groups  is well known. For our
purposes we will use the description from~\cite{bg}.

\begin{proposition}\label{prop2}{\cite[Proposition 2.1]{bg}}
Let $T$ be a maximal torus of $G=\SL_{n}^\varepsilon(q)$ corresponding to the
partition $n_1+n_2+\ldots+n_m=n$. Let $U$ be a subgroup of
$\GL_n(\overline{\F}_p)$ consisting of all diagonal matrices of the form
$\bd(D_1, D_2,\ldots,D_m)$, where $$D_i=\diag(\lambda_i,
\lambda_i^{\varepsilon q},\ldots, \lambda_i^{(\varepsilon q)^{n_i-1}})$$ and
$\lambda_i^{(\varepsilon q)^{n_i}}=1$ for $i\in\{1,2,\ldots,m\}$. Then $T_w$
and $U\cap\SL_n(\overline{\F}_p)$ are conjugate in~$\GL_n(\overline{\F}_p)$.
\end{proposition}

\begin{proposition}\label{BG_PSL}{\cite[Theorem 2.1]{bg}}
Let $n\geqslant 2$ and $n_1+n_2+\ldots+ n_s=ò$ be a partition, and choose a
maximal torus $T$  of $\SL_n(q)$ corresponding to the partition. For $1
\leqslant i \leqslant s$, define
$$d_i =\underset{1\leqslant j_1<\ldots<j_i\leqslant s} \lcm \gcd \{q^{n_{j_1}}-1, \ldots, q^{n_{j_i}}-1\}.$$
Then $d_i$ divides $d_{i'}$ for $i > i'$, and
$$T= d_1\times d_2\times\ldots\times d_{s-1}\times\frac{d_s}{q-1}$$
is the standard decomposition of $T$.
Let $\widetilde{T}$ be the image of $T$ in $\PSL_n(q)$. Put $d=(n,q-1)$ and
$d'=(n/(n_1,\ldots,n_s),q-1)$. Then
$$\widetilde{T}=\frac{d_1}{d(q-1)}$$ is the standard decomposition of
$\widetilde{T}$ if $s=1$, and
$$\widetilde{T}=d_1\times d_2\times\ldots\times d_{s-2}\times
\frac{d_{s-1}}{d'}\times\frac{d'd_s}{d(q-1)}$$ is the standard decomposition
of $\widetilde{T}$ if $s>1$.
\end{proposition}

\begin{proposition}\label{BG_PSU}{\cite[Theorem 2.2]{bg}}
Let $n\geqslant2$  and $n_1+n_2+\ldots+ n_s=ò$ be a partition, and consider a
maximal torus $T$ of  $\SU_n(q^2)$ corresponding to the partition. For
$1\leqslant i\leqslant s$, define
$$d_i =\underset{1\leqslant j_1<\ldots<j_i\leqslant s}\lcm\gcd\{q^{n_{j_1}}-(-1)^{n_{j_1}},\ldots, q^{n_{j_i}}-(-1)^{n_{j_i}}\}.$$  Then $d_i$ divides $d_{i'}$ for $i > i'$, and
$$T= d_1\times d_2\times\ldots\times d_{s-1}\times\frac{d_s}{q+1}$$
is the standard decomposition of $T$.
Let $\widetilde{T}$ be the image of $T$ in $\PSU_n(q^2)$. Put $d=(n,q+1)$ and
$d'=(n/(n_1,\ldots,n_s),q+1)$. Then
$$\widetilde{T}\simeq\frac{d_1}{d(q+1)}$$ is the standard decomposition of
$\widetilde{T}$ if $s=1$, and $$\widetilde{T}\simeq d_1\times
d_2\times\ldots\times d_{s-2}\times
\frac{d_{s-1}}{d'}\times\frac{d'd_s}{d(q+1)}$$ is the standard decomposition
of $\widetilde{T}$ if $s>1$.
\end{proposition}


If $G$ is equal to $\GL_n(q),\U_n(q),\SL_n(q),\SU_n(q)$, an element is semisimple regular if and only if its minimal
polynomial is equal to its characteristic polynomial, i.e. if and only if its characteristic polynomial has no repeated
roots (see \cite{fg}).

We briefly describe the general idea of proofs of the main results. We need to compute the  lower and upper bounds for
the ratio $\frac{|w(G)|}{|G|}$. We obtain the bounds only for regular semisimple elements, and then apply them to
obtain the bounds for the hole group. So we need to consider maximal tori of $G$. Since any regular semisimple element
belongs to a unique maximal torus, we can take then a union (and this union is disjoint) of tori, thus avoiding
repetition. Finally, instead of considering all maximal tori we could restrict ourselves by  a finite number (in our
paper one or two) maximal tori, and then multiply the bound by the index of normalizer.

\section{Example}
In this section, we give an example that the constant $c$ in the conjecture
depends on the word $\omega$.

\begin{lemma}\label{ss}
Assume that $G\simeq A_{n-1}^\varepsilon (q)$, and let $G_{ss}$ be the set of
all semisimple elements of $G$. Let $\omega$ be a power word $x^M$. Assume
that there exists $N$ such that the inequality
$|\omega(T)|\leqslant\frac{|T|}{N}$ holds for every maximal torus $T$ of $G$.
Then
$$|\omega(G_{ss})|\leqslant \frac{|G|}{N}.$$
\end{lemma}

\begin{proof}
Let $T_\lambda$ be a maximal torus corresponding to a partition $\lambda$,
and $\sigma_\lambda$ be a  representative of the conjugacy class of $\Sym_n$
corresponding to $\lambda$. Then $G_{ss}=\underset{\lambda}\bigcup\
T_\lambda^G$, where $\lambda$ runs through all partitions of $n$. Since
$$|N_G(T_\lambda)|\geqslant|N(G,T_\lambda)|=|T_\lambda|\cdot|C_W(\sigma_\lambda)|,$$
the number of conjugate tori in $T_\lambda^G$ equals
$\frac{|G|}{|N_G(T_\lambda)|}$ and is not greater than
$\frac{|G|}{|T_\lambda|\cdot|C_W(\sigma_\lambda)|}$. Hence, we have
\begin{multline*} |\omega(T_\lambda^G)|\leqslant \sum_{T\in T_\lambda^G}\vert
\omega(T)\vert\leqslant
\frac{|G|}{|T_\lambda|\cdot|C_W(\sigma_\lambda)|}|\omega(T_\lambda)|\leqslant  \\
\leqslant \frac{|G|}{|T_\lambda|\cdot|C_W(\sigma_\lambda)|}
\frac{|T_\lambda|}{N} = \frac{|G|}{N|C_W(\sigma_\lambda)|}.
\end{multline*}
Notice that
$$\sum_{\lambda} \frac{1}{|C_W(\sigma_\lambda)|} = \sum_{\lambda} \frac{|Cl(\sigma_{\lambda})|}{|\Sym_n|} =  \frac{1}{|\Sym_n|}\sum_{\lambda} |Cl(\sigma_{\lambda})| =1, $$
where $Cl(\sigma_{\lambda})$ is the conjugacy class of $\sigma_{\lambda}$ in
$\Sym_n$. Therefore,
$$
|\omega(G_{ss})|=|\omega(\bigcup_{\lambda} T_\lambda^G)| = |\bigcup_{\lambda} \omega(T_\lambda^G)| \leqslant \sum_\lambda |\omega(T_\lambda^G)|\leqslant \sum_\lambda \frac{|G|}{N|C_W(\sigma_\lambda)|}=\frac{|G|}{N}.$$
\end{proof}
The following lemma is clear.
\begin{lemma}\label{power}
Assume $T=d_1\times d_2\times\cdots\times d_{s}$ is the standard
decomposition of a finite abelian group $T$, and $\omega(x)=x^M$ is a power
word. Then
$$\omega(T) =\frac{d_1}{(M,d_1)}\times \frac{d_2}{(M,d_2)}\times\cdots\times
\frac{d_s}{(M,d_s)}$$ is the standard decomposition of $\omega(T)$.
\end{lemma}

\begin{lemma}\label{torusSL}
Let $n$ be a positive integer and  $p$ be a prime. Choose odd $l$ so that
$(l,n)=1$ (in particular $l$ could be equal to $1$). Assume that
$\widetilde{T}$ be a maximal torus of $\widetilde{G}=\PSL_n^\varepsilon
(p^l)$. Consider $\omega(x)=x^M$, where $M=p^n-(\varepsilon1)^n$. Then
$$|w(\widetilde{T})|\leqslant \frac{|\widetilde{T}|n}{p-1}.$$
\end{lemma}

\begin{proof}
Denote $p^l$ by $q$ for brevity. In view of Propositions \ref{BG_PSL} and
\ref{BG_PSU}, $\widetilde{T}$ corresponds to a partition of $n$. Let
$n=n_1+\ldots+n_s$ be the corresponding partition.

Assume that  $s=1$ first. Then

$$\widetilde{T}\simeq\frac{d_1}{d(q-\varepsilon 1)}=\frac{q^n-(\varepsilon 1)^n}{d(q-\varepsilon
1)},$$

where $d=(n,q-\varepsilon1)$.
By Lemma~\ref{power} we have

$$\omega(\widetilde{T})\simeq\frac{q^n-(\varepsilon 1)^n}{d(q-\varepsilon 1)\left(M,\frac{q^n-(\varepsilon 1)^n}{d(q-\varepsilon 1)}\right)}.$$

Notice that for $a,b,c\in\mathbb{N}$ with $b$ dividing $a$, we have
$(\frac{a}{b},c)\geqslant\frac{1}{b}(a,c)$. Therefore

$$\left( M,\frac{q^n-(\varepsilon1)^n}{d(q-\varepsilon1)}\right)= \left(p^n-(\varepsilon1)^n,\frac{p^{ln}-(\varepsilon1)^n}{(n,q-\varepsilon1)(p^l-\varepsilon1)}\right)\geqslant
\frac{1}{n}\left({p^n-(\varepsilon1)^n},\frac{p^{ln}-(\varepsilon1)^n}{p^l-\varepsilon1}\right).$$

Since $(l,n)=1$ By Lemma~\ref{Zav}, both $p^n-(\varepsilon 1)^n$ and
$p^l-\varepsilon1$ divide $p^{ln}-(\varepsilon 1)^n$, so $[p^n-(\varepsilon
1)^n,p^l-\varepsilon1]$ divides  $p^{ln}-(\varepsilon 1)^n$. Since $(l,n)=1$
Lemma~\ref{Zav} implies, that $(p^n-(\varepsilon 1)^n,p^l-\varepsilon1)$
divides $p-\varepsilon 1$, so $\frac{(p^n-(\varepsilon
1)^n)(p^l-\varepsilon1)}{p-\varepsilon 1}$ divides $p^{ln}-(\varepsilon
1)^n$. Therefore

$$\left( M,\frac{q^n-(\varepsilon1)^n}{d(q-\varepsilon1)}\right) \geqslant
\frac{1}{n}\left(p^n-(\varepsilon1)^n,\frac{p^{ln}-(\varepsilon1)^n}{p^l-\varepsilon1}\right)
\geqslant \frac{1}{n}\cdot \frac{p^n-(\varepsilon1)^n}{p-\varepsilon1}.$$

 Hence,
$$|w(\widetilde{T})|=\frac{|\widetilde{T}|}{\left(M,\frac{q^n-(\varepsilon1)^n}{d(q-\varepsilon 1)}\right)}
\leqslant \frac{|\widetilde{T}|n(p-\varepsilon1)}{p^n-(\varepsilon1)^n} \leqslant \frac{|\widetilde{T}|n}{p-1}.$$
If $s>1$ then
$$\widetilde{T}\simeq d_1\times d_2\times\ldots\times d_{s-2}\times \frac{d_{s-1}}{d'}\times\frac{d'd_s}{d(q-\varepsilon 1)}.$$
It follows from the defenition of $d_i$ that $q-\varepsilon 1$ divides $d_i$
for all $1\leqslant i\leqslant s$. By Lemma \ref{Zav},
$$(M,d_i)\geqslant(p^n-(\varepsilon 1)^n,q-\varepsilon 1)\geqslant p-1, \quad (M,\frac{d_{s-1}}{d'}) \geqslant \frac{1}{n}(M,d_{s-1})\geqslant \frac{p-1}{n}.$$
By Lemma~\ref{power} we have
\begin{eqnarray*}
|\omega(\widetilde{T})| &\leqslant&  \frac{d_1}{(M,d_1)}\cdots \frac{d_{s-2}}{(M,d_{s-2})} \cdot\frac{d_{s-1}}{d'(M,d_{s-1}/d')}\cdot\frac{d'd_s}{d(q-\varepsilon 1)} \\
&\leqslant& \frac{d_1}{p-1}\cdots \frac{d_{s-2}}{p-1} \cdot\frac{d_{s-1}n}{d'(p-1)} \cdot \frac{d'd_s}{d(q-1)} = \frac{|\widetilde{T}|n}{(p-1)^{s-1}} \leqslant \frac{|\widetilde{T}|n}{p-1}.
\end{eqnarray*}
\end{proof}

Thus we have the following.
\begin{theorem}\label{main}
Let $n$ be a positive integer and  $p$ be a prime. Choose odd $l$ so that
$(l,n)=1$ (in particular $l$ could be equal to $1$). Assume that $T$ is a
maximal torus of $G=\PSL_n^\varepsilon (p^l)$. Consider $\omega(x)=x^M$,
where $M=p^n-(\varepsilon1)^n$. Then
$$|\omega(G)| \leqslant  \frac{4|G|n}{p-1}.$$
\end{theorem}

\begin{proof}
By Lemma~\ref{torusSL} we have $|\omega(T)|\leqslant\frac{|T|n}{p-1}$ for any
maximal torus $T$ of $G$. By Lemma~\ref{ss}
$$|\omega(G_{ss})|\leqslant \frac{|G|}{(p-1)/n}=\frac{|G| n}{p-1}.$$
According to {~\cite[Theorem 1.1]{gl}}, we have
$$|G_{ss}| \geqslant |G| \left(1-\frac{3}{q-1}-\frac{2}{(q-1)^2}\right),$$
so $\vert G\setminus G_{ss}\vert\leqslant \frac{5}{q-1}\vert G\vert$. Now we have
$$|\omega(G)|\leqslant \vert\omega(G_{ss})\vert+\vert G\setminus G_{ss}\vert \leqslant \frac{|G|n}{p-1}+\frac{5}{p-1}\vert G\vert < \frac{4|G|n}{p-1}.$$
\end{proof}

As a corollary to Theorem~\ref{main} we get an example showing that the
constant $c$ in the Conjecture does depend on $\omega$.
\begin{example} Let $G\simeq \PSL_n^\varepsilon(q)$ be simple.
For a constant $c$ we choose a prime $p$ such that $p-1>\frac{4n^2}{c}$. Let
$\omega(x)=x^M$, where $M=p^{n}-(\varepsilon1)^n$. For any number $N =
N(\omega)$ there exist infinitely many odd numbers $l$ such that $(l,n)=1$
and $|\PSL_n^\varepsilon(p^l)|\geqslant N$. Then by Theorem~\ref{main} we
have
$$|\omega(G)|\leqslant\frac{4|G|n}{p-1}<\frac{c|G|}{n}.$$
\end{example}

\section{Estimate for power word in $\SL_n^\varepsilon (q)$}
The following lemma shows that it is enough to prove the Theorem~\ref{th1} for $\SL_n^\varepsilon(q)$.
\begin{lemma}
Let $\pi \colon G \rightarrow H$ be a surjective homomorphism of finite
groups and $\omega$ be a word. Then,
$$\frac{|\omega(H)|}{|H|} \geqslant \frac{|\omega(G)|}{|G|}.$$
\end{lemma}

\begin{proof}
As $\pi$ is a homomorphism, for $g_1, g_2, \dots, g_d \in G$ it is straightforward to check that
$\pi(\omega(g_1, g_2, \dots, g_d)) = \omega(\pi(g_1), \pi(g_2), \dots, \pi(g_d))$.
Thus we have $\omega(g_1, g_2, \dots, g_d) \in \pi^{-1}(\omega(\pi(g_1), \pi(g_2), \dots, \pi(g_d)))$. Hence the inclusion
$\omega(G) \subset \pi^{-1}(\omega(H))$ holds, and we have
$|\omega(G)| \leqslant |\ker(\pi)|\cdot|\omega(H)|$.
Thus $\frac{|\omega(H)|}{|H|} \geqslant \frac{|\omega(G)|}{|\ker(\pi)|\cdot|H|} = \frac{|\omega(G)|}{|G|}$.
\end{proof}

For any subset $S$ of a maximal torus denote by $S_{reg}$ the set  of regular
semisimple elements in $S$.

\begin{lemma}\label{isotropic}
For every  $\omega=x^M$ there exists a number $N=N(\omega)$ such that if $T$
is a maximal torus of $G=\SL_n^\varepsilon(q)$ corresponding to the partition
$n$, then
$$ |\omega(T)_{reg}| \geqslant \frac{|T|}{2M}$$
for $|G|\geqslant N$.
\end{lemma}

\begin{proof}
Consider the subgroup $$U=\{\diag(\lambda, \lambda^{\varepsilon q},\ldots,
\lambda^{(\varepsilon q)^{n-1}})\mid \lambda^{\frac{q^n-(\varepsilon
1)^n}{q-\varepsilon1}}=1\}$$ of $\GL_n(\overline{\F}_p)$. Then $T$ and $U$
are conjugate by Proposition~\ref{prop2}. Recall that an element
$\diag(\lambda_1,\ldots,\lambda_n)$ of $\GL_n(\overline{\F}_p)$ is regular if
and only if $\lambda_i\neq \lambda_j$ for all $i\neq j$.

If $n=2$, then $U$ is a cyclic group of order $q+\varepsilon1$ and
$U\setminus U_{reg}$ consists of $E$ and $-E$, where $E$ is the identity
matrix. So
$$\vert \omega(U)_{reg}\vert\geqslant \vert \omega(U)\vert -2\geqslant
\frac{\vert U\vert}{M}-2\geqslant \frac{\vert U\vert}{2M},$$ if $\vert
G\vert=\frac{1}{(2,q-1)}q(q^2-1)$ is big enough. Thus there exists $N_2$ such
that if $\vert G\vert\geqslant N_2$, then the inequality $ |\omega(T)_{reg}|
\geqslant \frac{|T|}{2M}$ holds.

If $n=3$, then $U$ is a cyclic group of order $q^2+\varepsilon q+1$ and
$U\setminus U_{reg}$ consists of $E$, $\lambda E$, and $\lambda ^2 E$, where
$\lambda$ is a cubic root of the unity in $\F_q$. So
$$\vert \omega(U)_{reg}\vert\geqslant \vert \omega(U)\vert -3\geqslant
\frac{\vert U\vert}{M}-3\geqslant \frac{\vert U\vert}{2M},$$ if $\vert
G\vert=\frac{1}{(2,q-\varepsilon1)}q^3(q^2-1)(q^3-\varepsilon 1)$ is big
enough. Thus there exists $N_3$ such that if $\vert G\vert\geqslant N_3$,
then the inequality $ |\omega(T)_{reg}| \geqslant \frac{|T|}{2M}$ holds.

Now assume that $n\geqslant 4$. Suppose $u=\diag(\lambda,
\lambda^{\varepsilon q},\ldots, \lambda^{(\varepsilon q)^{n-1}})\in U$ and
$u^M$ is not regular. Then $\lambda^{M(\varepsilon
q)^k}=\lambda^{M(\varepsilon q)^l}$ for some $n-1\geqslant k>l\geqslant 0$.
Thus, we have

\begin{equation}\label{eq} \lambda^{M(\varepsilon q)^l((\varepsilon q)^{r}-1)}=1 \quad \text{ for }
r=k-l.
\end{equation}

Since $|\lambda|$ divides $|U|=1+\varepsilon q+\ldots+(\varepsilon q)^{n-1}$,
we have $(|\lambda|, q)=1$. Therefore $\vert \lambda\vert=\vert
\lambda^{(\varepsilon q)^l}\vert$  and $\lambda^{M((\varepsilon q)^{r}-1)}=1$.
Hence, applying Lemma \ref{Zav}, we obtain that  the number of solutions
to~(\ref{eq}) is equal to
$$\left(\frac{q^n-(\varepsilon 1)^n}
{q-\varepsilon 1}, M((\varepsilon q)^r-1)\right)\leqslant
M(q^n-(\varepsilon 1)^n,(\varepsilon q)^r-1)\leqslant
M(q^{(n,r)}+1).$$ Further, since
$(r,n)=(n-r,n)\leqslant n/2$, we have $q^{(r,n)}+1< 2
q^{n/2}.$ Then the total number of elements $u\in U$ such that $u^M$ is not
regular is not greater than
$$\sum_{0\leqslant l<k\leqslant n-1}M(q^{(n,k-l)}+1)\leqslant \sum_{0\leqslant l<k\leqslant n-1}2Mq^{(n,k-l)}
\leqslant  n^2Mq^{n/2}.$$
It follows that
\begin{equation}\label{boundRegSinger}
|\omega(U)_{reg}|\geqslant|\omega(U)|-n^2Mq^{n/2}\geqslant\frac{|U|}{M}-n^2Mq^{n/2}.
\end{equation}
We show that there exists $N$ such that
\begin{equation}\label{boundSinger}
n^2Mq^{n/2}<\frac{|U|}{2M}
\end{equation}
if $\vert G\vert\geqslant N$. If \eqref{boundSinger} is true, then from
\eqref{boundRegSinger} we obtain $|\omega(U)_{reg}|>\frac{|U|}{2M}$.

Thus we remain to show that there exists $N$ such that for every $G$ with
$\vert G\vert\geqslant N$ inequality \eqref{boundSinger} holds. First notice
that $$\vert U\vert =q^{n-1}+\varepsilon q^{n-2}+\varepsilon^2 q^{n-3}+\ldots
>q^{n-1}+\varepsilon q^{n-2}\geqslant q^{n-2}(q-1),$$ so inequality
\eqref{boundSinger} would follow from
\begin{equation}\label{boundSinger1}
\frac{q^{n-n/2-2}(q-1)}{2M}>n^2 M.
\end{equation}
Since $q\geqslant 2$, inequality \eqref{boundSinger1} holds, if
$2^{n/2}>8n^2M^2$. Clearly, there exists $n_0$ such that for every
$n\geqslant n_0$ the last inequality holds.  Thus for every $G$ of rank at
least $n_0$ inequality \eqref{boundSinger1} holds.

If $n<n_0$, then clearly for every $n$ there  exists $q_n$ such that for
every $q\geqslant q_n$ inequality \eqref{boundSinger1} holds. Denote
$q_n^{2n^2}$ by $N_n$. Since $\vert G\vert< q^{2n^2}$, we obtain that if $G$
is of rank $n$ and  $\vert G\vert\geqslant N_n$, then $q>q_n$, so inequality
inequality \eqref{boundSinger1} holds. Denote $\max\{N_n\mid 2\leqslant
n\leqslant n_0\}$ by $N$. If $\vert G\vert\geqslant N$, then either
$n\geqslant n_0$, or $\vert G\vert>N_n$. In both cases inequality
\eqref{boundSinger1} holds and the lemma follows.
\end{proof}

\begin{lemma}\label{torus_i+n-i}
For every  $\omega=x^M$ there exists $N=N(\omega)$ such that for every
maximal torus $T$ corresponding to a partition $i+(n-i)$, where $1\leqslant
i\leqslant n-1$, of  $G=\SL_n^\varepsilon (q)$ the inequality
$$ |\omega(T)_{reg}| \geqslant \frac{|T|}{2M^2}$$
 holds if $|G|\geqslant N$.
\end{lemma}

\begin{proof}
We use similar arguments as in the proof of Lemma~\ref{isotropic}. Let
\begin{multline*}
U=\{\diag(\lambda, \lambda^{\varepsilon q},\ldots,\lambda^{(\varepsilon
q)^{i-1}}, \mu,\mu^{\varepsilon q},\ldots, \mu^{(\varepsilon q)^{n-i-1}})
\mid \\ \lambda^{{q^i-(\varepsilon 1)^i}}= \mu^{{q^{n-i}-(\varepsilon 1)^{n-i}}}=1;
\lambda^{\frac{q^i-(\varepsilon 1)^i}{q-\varepsilon 1}}\cdot
\mu^{\frac{q^{n-i}-(\varepsilon 1)^{n-i}}{q-\varepsilon 1}}=1\}
\end{multline*}
be a subgroup of $\GL_n(\overline{\F}_p)$. Then by Proposition~\ref{prop2} subgroups  $T$ and $U$ are conjugate in
$\GL_n(\overline{\F}_p)$. We want to bound $\vert \omega(U)_{reg}\vert $. Suppose $u\in U$ and $u^M$ is not regular. We
have
$$u^M=\diag(\lambda^M,\lambda^{M\varepsilon q},\ldots,\lambda^{M\varepsilon q^{i-1}},
\mu,\mu^{M\varepsilon q},\ldots,\mu^{M\varepsilon q^{n-i-1}}).$$
Then we have the following three possibilities
\begin{equation}\label{eq1}
 \lambda^{M(\varepsilon q)^k}=\lambda^{M(\varepsilon q)^l}\quad\text{for}\quad 0\leqslant l<k\leqslant i-1,
\end{equation}
\begin{equation}\label{eq2}
\mu^{M(\varepsilon q)^k}=\mu^{M(\varepsilon q)^l}\quad\text{for}\quad 0\leqslant l<k \leqslant n-i-1,
\end{equation}
\begin{equation}\label{eq3}
\lambda^{M(\varepsilon q)^k}=\mu^{M(\varepsilon q)^l}\quad\text{for}\quad 0\leqslant k\leqslant i-1, 0\leqslant l\leqslant n-i-1.
\end{equation}
Applying the same arguments as in the proof of Lemma~\ref{isotropic} we could deal with
the cases~(\ref{eq1}) and~(\ref{eq2}). We have that the total number of elements $u\in U$ such that $u^M$ is not regular in these two cases is not greater than $2n^2Mq^{n/2}$.

In the case~\eqref{eq3}   without loss of generality we can assume
$i\geqslant n-i$. Rising to the $(\varepsilon q)^{i-k}$ power both parts of
\eqref{eq3}, we obtain
$$\mu^{M(\varepsilon q)^{l+i-k}}=\lambda^{M(\varepsilon q)^i}=\lambda^M.$$
Therefore, $(\lambda\mu^{-(\varepsilon q)^{l+i-k}})^M=1$. The equation
$x^M=1$ has at most $M$ solutions in $\overline{\F}_p$. Hence, for each fixed
element $\mu$ we have at most $M$ possibilities for $\lambda$. In the torus
$U$ we have $q^{n-i}-(\varepsilon 1)^{n-i}$ possibilities for $\mu$. Hence,
the number of elements $u=u(\lambda,\mu)\in U$ such that $u^M$ is not regular
in the case~(\ref{eq3}) is at most $M( q^{n-i}-(\varepsilon
1)^{n-i})\leqslant 2Mq^{n/2}$. The total number  of such elements is not
greater than
$$\sum_{k = 0}^{i-1}\sum_{l = 0}^{n-i-1}2Mq^{n/2}\leqslant n^2Mq^{n/2}.$$
Finally, we have
$$|\omega(U)_{reg}|\geqslant|\omega(U)|-2n^2Mq^{n/2}-n^2Mq^{n/2}=|\omega(U)|-3n^2Mq^{n/2}.$$
As in the proof of Lemma~\ref{isotropic} we can find $N$ such that
$\frac{\vert U\vert}{2M^2}>3n^2Mq^{n/2}$ if $\vert G\vert\geqslant N$.  So we
have
$$|\omega(U)_{reg}|\geqslant\frac{|U|}{M^2}-3n^2Mq^{n/2} >\frac{|U|}{2M^2}$$
for   $|G|\geqslant N$.
\end{proof}

\begin{lemma}\label{conjugates_i+n-i}
For every  $\omega=x^M$ there exists $N=N(\omega)$ such that for every
maximal torus $T$ corresponding to a partition $i+(n-i)$, where $1\leqslant
i\leqslant n-1$, of  $G=\SL_n^\varepsilon (q)$ the inequality
$$|\omega(T^G)_{reg}|\geqslant \frac{|G|}{2M^2|W(T)|}$$ holds if $|G|\geqslant N$.
\end{lemma}
\begin{proof}
Take $N$ from the conclusion of Lemma \ref{torus_i+n-i}. Then $
|\omega(T)_{reg}| \geqslant \frac{|T|}{2M^2}$, in particular
$T_{reg}\not=\varnothing$. By Lemma \ref{NomMaxTorus} we have
$|N_G(T)|=|N(G,T)|=|T|\cdot|W(T)|$. Then the number $k$ of conjugates of
torus $T$ in $G$ is equal to
$\frac{|G|}{|N_G(T)|}=\frac{|G|}{|T|\cdot|W(T)|}$.

Since $\omega(T_1\cup T_2)=\omega(T_1)\cup \omega(T_2)$ and $(T_1\cup T_2)_{reg}=(T_1)_{reg}\cup (T_2)_{reg}$, we have
$$
|\omega(T^G)_{reg}|=|(\omega(\bigcup_{i=1}^kT^{g_i}))_{reg}|=
|(\bigcup_{i=1}^k\omega(T^{g_i}))_{reg}|=
|\bigcup_{i=1}^k(\omega(T^{g_i}))_{reg}|,
$$
where $g_1,\ldots,g_k$ is the right transversal of $N_G(T)$. Since every regular semisimple element lies only in one
maximal torus, it follows that
$$|\bigcup_{i=1}^k(\omega(T^{g_i}))_{reg}|=\sum\limits_{i=1}^k|\omega(T^{g_i})_{reg}|=\sum\limits_{i=1}^k|\omega(T)_{reg}|.$$

Applying Lemma~\ref{torus_i+n-i} we get
$$|\omega(T^G)_{reg}|=\sum\limits_{i=1}^k|\omega(T)_{reg}|\geqslant
\frac{|G|}{|T||W(T)|}\cdot\frac{|T|}{2M^2}=\frac{|G|}{2M^2|W(T)|},$$ and the lemma follows.
\end{proof}

Now we are ready to give an affirmative answer to the conjecture.

\begin{theorem}\label{1torus}
If $G\simeq \SL_n^\varepsilon(q)$, then for every  $\omega=x^M$ there exists
$N=N(\omega)$ such that
$$|\omega(G)| \geqslant \frac{1}{2nM} |G|$$
for $|G|\geqslant N$.
\end{theorem}

\begin{proof}
Let $T$ be a maximal torus of $G$ corresponding to  a partition $n$, and choose $N$ from the conclusion of Lemma
\ref{isotropic}. Since the inequality $ |\omega(T)_{reg}| \geqslant \frac{|T|}{2M}$ holds, $T$ contains regular
elements. By Lemma \ref{NomMaxTorus}, we have $|N_G(T)|=|N(G,T)|=|T|\cdot|W(T)|=\vert T\vert\cdot n$. Then applying
Lemma~\ref{isotropic} and arguments from the proof of the Lemma \ref{conjugates_i+n-i}, we obtain
$$|\omega(G)| \geqslant |\omega(T^G)| \geqslant \frac{|G|}{2M|W(T)|} = \frac{|G|}{2nM}$$
and the theorem follows.
\end{proof}

\section{Proof of Theorem~\ref{th1} and Theorem~\ref{th2}}

{\bf Proof of Theorem~\ref{th1}.} Let $T_0$ be a maximal torus of $G$ corresponding to a partition $n$. It follows from the proof of Theorem~\ref{1torus} that there exists $N_1=N_1(\omega)$ such that $|\omega(T_0^G)_{reg}| \geqslant \frac{|G|}{2M n}$, if $|G|\geqslant N_1$.
Let $k=[\frac{n}{2}]$ and $T_i$ be a maximal torus of $G$ corresponding to a partition $i+(n-i)$ for $i=1,\ldots,k$. By Lemma~\ref{conjugates_i+n-i} there exists $N_2=N_2(\omega)$
such that $|\omega(T_i^G)_{reg}| \geqslant \frac{|G|}{2M^2 |W(T_i)|}$ for all
$i=1,\ldots,k$, if $|G|\geqslant N_2$.

Let $N=\max\{N_1, N_2\}$. Since every regular semisimple element lies only in one maximal torus, we have
$$|\omega(G)|\geqslant|\omega(\bigcup\limits_{i=0}^kT_i^G)_{reg}|=
|\bigcup\limits_{i=0}^k(\omega(T_i^G))_{reg}|=\sum\limits_{i=0}^k|\omega(T_i^G)_{reg}|\geqslant
\sum\limits_{i=1}^k\frac{|G|}{2M^2|W(T_i)|}+\frac{|G|}{2M n}.$$ Let $H_n=\sum\limits_{k=1}^n
\frac{1}{k}$ be $n^{th}$ harmonic number. If $n$ is odd, then
$$\sum_{i=1}^{k} \frac{1}{|W(T_i)|}=\sum_{i=1}^{k} \frac{1}{i(n-i)}=\frac{1}{n} \sum_{i=1}^{k} \left(\frac{1}{i}+\frac{1}{n-i}\right)= \frac{1}{n}H_{n-1}.$$
If $n$ is even, then
$$\sum_{i=1}^{k} \frac{1}{|W(T_i)|}=\sum_{i=1}^{k-1} \frac{1}{i(n-i)}+\frac{1}{2k^2} =\frac{1}{n} \sum_{i=1}^{k-1} \left(\frac{1}{i}+\frac{1}{n-i}\right)+\frac{1}{nk}= \frac{1}{n}H_{n-1}.$$
Now we have
$$|\omega(G)|\geqslant\frac{|G|}{2M^2}\left(\sum\limits_{i=1}^k\frac{1}{|W(T_i)|}+\frac{M}{n}\right)=
\frac{|G|}{2M^2}\left(\frac{H_{n-1}}{n}+\frac{M}{n}\right)\geqslant
\frac{|G|}{2M^2}\frac{\ln(n)}{n},$$
as claimed.
\vspace{0.5em}

{\bf Proof of Theorem~\ref{th2}.}
As we mentioned in the proof of Theorem~\ref{main}, we have
$|\omega(G)|\leqslant \vert\omega(G_{ss})\vert+\vert G\setminus G_{ss}\vert \leqslant \vert\omega(G_{ss})\vert+\frac{5}{q-1}\vert G\vert.$ Therefore, it is enough to prove that
$|\omega(G_{ss})|<c\vert G\vert$.

Let $T_\lambda$ be a maximal torus of $G$ corresponding to a partition $\lambda=\{n_1,\ldots,n_k\}$ of $n$.
Then by Proposition~\ref{BG_PSL} we have
$$T\simeq d_1\times d_2\times\ldots\times d_{k-1}\times\frac{d_k}{q-1}.$$
By Lemma~\ref{power} we have
\begin{eqnarray*}
|\omega(T_\lambda)| &=&  \frac{d_1}{(M,d_1)}\cdots \cdot\frac{d_{k-1}}{(M,d_{k-1})}
\cdot\frac{d_k}{(q-1)(M,\frac{d_k}{q-1})} \\
&\leqslant& \frac{d_1}{q-1}\cdots \cdot\frac{d_{k-1}}{q-1}
\cdot\frac{d_k}{q-1}=\frac{|T_\lambda|}{(q-1)^{k-1}}.
\end{eqnarray*}
Further,
$$|\omega(T_\lambda^G)|\leqslant\frac{|G|}{|N_G(T_\lambda)|}|\omega(T_\lambda)|\leqslant \frac{|G|}{|T_\lambda||W(T_\lambda)|}|\omega(T_\lambda)|\leqslant \frac{|G|}{|W(T_\lambda)|(q-1)^{k-1}}.$$
We have $G_{ss}=\underset{\lambda}\bigcup\ T_\lambda^G$, where $\lambda$ runs through all partitions of $n$. Hence,
$$
|\omega(G_{ss})|=|\omega(\bigcup_{\lambda} T_\lambda^G)|\leqslant \sum_\lambda |\omega(T_\lambda^G)|\leqslant \sum_\lambda \frac{|G|}{|W(T_\lambda)|(q-1)^{k-1}}.$$
We could rewrite our expression as follows
$$\sum\limits_{\lambda} \frac{1}{|W(T_\lambda)|(q-1)^{k-1}}=
\sum\limits_{k=1}^n\frac{1}{(q-1)^{k-1}}\sum\limits_{\lambda_k}\frac{1}{W(T_{\lambda_k})},
$$
where the second sum is taking over all partitions $\lambda_k$ of $n$ into exactly $k$ parts.
Let  $s(n,k)$ denote the Stirling numbers of the first kind. Then
$\sum\limits_{\lambda_k}\frac{1}{W(T_{\lambda_k})}=\frac{s(n,k)}{n!}$. It follows from~\cite{hw} that $\frac{s(n,k)}{n!}\rightarrow 0$, as $k=o(\ln n)$ and $n\rightarrow\infty$.
Let $r=\sqrt{\ln n}=o(n)$, $s(n)=\max\{\frac{s(n,1)}{n!},\frac{s(n,2)}{n!},\ldots,\frac{s(n,r)\}}{n!}\}$ and $m=[\frac{c}{2s(n)}]$. Then $m\rightarrow \infty$, as $n\rightarrow\infty$. Therefore,
\begin{multline*}
|\omega(G_{ss})|\leqslant\sum\limits_{k=1}^n\frac{1}{(q-1)^{k-1}}\sum\limits_{\lambda_k}\frac{1}{W(T_{\lambda_k})}\leqslant\\
\leqslant\sum\limits_{k=1}^{m-1}\frac{1}{(q-1)^{k-1}}\sum\limits_{\lambda_k}\frac{1}{W(T_{\lambda_k})}+
\sum\limits_{k=1}^n\frac{1}{(q-1)^{m}}\sum\limits_{\lambda_k}\frac{1}{W(T_{\lambda_k})}\leqslant\\
\leqslant\sum\limits_{k=1}^{m-1}\frac{1}{(q-1)^{k-1}}s(n)+
\frac{1}{(q-1)^{m}}\leqslant s(n)m+\frac{1}{(q-1)^{m}}\leqslant \frac{c}{2}+\frac{1}{(q-1)^{m}}<c,
\end{multline*}
for $m$ big enough. The theorem is proved for $\SL_n(q)$.
For $\PSL_n(q)$ it is enough to choose $n$ such that $(n,q-1)=1$, which implies $\SL_n(q)\simeq\PSL_n(q)$.


\end{document}